\documentclass[11pt,twoside,a4paper]{article}

\usepackage[T1]{fontenc}

\usepackage{amsmath,amssymb}

\usepackage[american]{babel}
\usepackage[useregional,showdow=false]{datetime2}
\DTMlangsetup[en-US]{ord=omit,showdayofmonth=true}

\usepackage{xurl}

\usepackage{enumitem}

\usepackage{graphicx}

\providecommand\grad{}
\renewcommand*\grad{\mathop{}\mathopen{\nabla}}
\providecommand\fun{}
\renewcommand*{\fun}[1]{\mathop{{}#1}\nolimits}
\providecommand\diff{}
\renewcommand*{\diff}{\mathop{{}d}\mathopen{}}

\newcommand\mytheorems{
	\usepackage{amsthm}
	\theoremstyle{plain}
	\newtheorem{theorem}{Theorem}
	\newtheorem{lemma}{Lemma}
	\newtheorem{proposition}{Proposition}
	\newtheorem{assumption}{Assumption}
	\newtheorem{corollary}{Corollary}
}

\newcommand{\myalgorithmtheorembegin}[1]{\begin{algorithm}[#1] \leavevmode}
\newcommand\myalgorithmtheoremend{\end{algorithm}}

\newcommand{\myalgorithmfloatbegin}[1]{\begin{algorithm} \caption{#1}}
\newcommand\myalgorithmfloatend{\end{algorithm}}

\makeatletter
\@ifclassloaded{article}{
	\def\blackslug{\hbox{\hskip 1pt \vrule width 4pt height 6pt depth 1.5pt
	  \hskip 1pt}}
	\usepackage{artmods}
}{}

\@ifundefined{proofof}{
	
}{}

\@ifpackageloaded{hyperref}{}{\usepackage[hidelinks]{hyperref}}

\date{\today}
\newcommand*\titlefull{On the asymptotic behavior of a higher-order extrapolation primal--dual interior-point method for nonlinear programming}
\newcommand*\titlerunninghead{Extrapolation primal--dual interior-point method}
\newcommand*\AMSyear{2020}
\newcommand*\AMSlongvalue{Primary 65K05, 90C30, 90C51.}
\newcommand*\AMSvalue{65K05, 90C30, 90C51}

\newcommand*\abstracttext{A trajectory-following primal--dual interior-point method solves nonlinear optimization problems with inequality and equality constraints by approximately finding points satisfying perturbed Karush--Kuhn--Tucker optimality conditions for a decreasing order of perturbation controlled by the barrier parameter. Under some conditions, there is a unique local correspondence between small residuals of the optimality conditions and points yielding that residual, and the solution on the barrier trajectory for the next barrier parameter can be approximated using an approximate solution for the current parameter. A framework using higher-order derivative information of the correspondence is analyzed in which an extrapolation step to the trajectory is first taken after each decrease of the barrier parameter upon reaching a sufficient approximation. It suffices asymptotically to only take extrapolation steps for convergence at the rate the barrier parameter decreases with when using derivative information of high enough order. Numerical results for quadratic programming problems are presented using extrapolation as accelerator.}
\newcommand*\keywordstext{interior-point methods, extrapolation methods, higher-order methods, local convergence, nonlinear programming}

\@ifclassloaded{article}{
	\pagestyle{headings}
	\textheight     21.5900cm
	\textwidth      13.9700cm
	\topmargin       0.5000cm
	\oddsidemargin   1.0700cm
	\evensidemargin  0.8700cm

	\title{\titlefull}
	\markboth{\titlerunninghead}{}
	\author{Pim Heeman\thanks{\raggedright \footAF} \and Anders Forsgren\addtocounter{footnote}{-1}\footnotemark}
	\def\footAF{Division of Numerical Analysis, Optimization and Systems Theory, Department of Mathematics, KTH Royal Institute of Technology, SE-100\;44 Stockholm, Sweden ({\tt pimh@kth.se;https://orcid.org/0009-0006-1948-6831, andersf@kth.se;https://orcid.org/0000-0002-6252-7815}).}

	\newcommand\mymaketitle{
		\maketitle
		\thispagestyle{empty}
		\ifdefined\abstracttext
		\begin{abstract}
			\abstracttext
		\end{abstract}
		\fi
	}

	\@ifpackageloaded{artmods}{
		\ifx\AMSvalue\empty\else
			\expandafter\def\expandafter\mymaketitle\expandafter{\mymaketitle \begin{keywords} \keywordstext \end{keywords} \begin{AMS} \AMSvalue \end{AMS}}
		\fi
	}{
		\mytheorems
		\newtheorem{algorithm}{Algorithm}
	}
	\newenvironment{myalgorithm}[1]{\myalgorithmtheorembegin{#1}}{\myalgorithmtheoremend}
}{}
\@ifclassloaded{amsart}{
	\expandafter\expandafter\expandafter\title\expandafter\expandafter\expandafter[\expandafter\titlerunninghead\expandafter]\expandafter{\titlefull}

	\author{Pim Heeman}
	\address{Division of Numerical Analysis, Optimization and Systems Theory, Department of Mathematics, KTH Royal Institute of Technology, SE-100 44 Stockholm, Sweden}
	\email{pimh@kth.se}

	\author{Anders Forsgren}
	\address{Division of Numerical Analysis, Optimization and Systems Theory, Department of Mathematics, KTH Royal Institute of Technology, SE-100 44 Stockholm, Sweden}
	\email{andersf@kth.se}

	\newcommand\mymaketitle{
		\ifdefined\abstracttext
		\begin{abstract}
			\abstracttext
		\end{abstract}
		\fi
		\maketitle
	}

	\@namedef{subjclassname@\AMSyear}{\textup{\AMSyear} Mathematics Subject Classification}
	\ifx\AMSlongvalue\empty\else \subjclass[\AMSyear]{\AMSlongvalue} \fi

	\mytheorems
	\newtheorem{algorithm}{Algorithm}
	\newenvironment{myalgorithm}[1]{\myalgorithmtheorembegin{#1}}{\myalgorithmtheoremend}
}{}
\makeatother

\begin{document}
\mymaketitle

\section{Introduction}

In this work, the asymptotic behavior of a primal--dual interior-point method framework that uses higher-order derivative information will be studied. Within the scope are general nonlinear continuous optimization problems with inequality and equality constraints of the form
\begin{align}
	\begin{split}
	\operatorname*{minimize}_{x \in \mathbb{R}^n}\quad &\mathopen{}\fun{f}(x) \\
	{\operatorfont subject\ to} \quad &\mathopen{}\fun{c_\mathcal{I}}(x) \geq 0, \\
	&\mathopen{}\fun{c_\mathcal{E}}(x) = 0,
	\end{split} \label{eqn:nlp}
\end{align}
with $\fun{f} \colon \mathbb{R}^n \to \mathbb{R}$ and $\fun{c_\mathcal{I}}(x)$~and~$\fun{c_\mathcal{E}}(x)$ referring to vectors of length $m_\mathcal{I}$~and~$m_\mathcal{E}$ respectively, where the $i$th element of the vector $\fun{c}(x) \triangleq \begin{pmatrix} \fun{c_\mathcal{I}}(x)^T & \fun{c_\mathcal{E}}(x)^T \end{pmatrix}^T$ of length~$m$ is the function~$\fun{c_i} \colon \mathbb{R}^n \to \mathbb{R}$ evaluated at~$x$. We let~$d$ be the smallest number of times each of the $(m + 1)$~functions $f$~and~$\fun{c_i}$ is continuously differentiable, and assume $d \geq 3$, i.e., that each function is at least thrice continuously differentiable.

\subsection{Notation}

We denote by~$[\cdot]_i$ the $i$th row of the matrix this notation is applied to and by~$[\cdot]_S$ the rows of the matrix indexed by index set~$S$ stacked on top of each other. As shorthand notation, we write $w \triangleq (x, \lambda)$ for the vector that stacks~$x \in \mathbb{R}^n$ on top of~$\lambda \in \mathbb{R}^m$ on the understanding that any symbols or arguments applied to~$w$ should also be applied to $x$ and $\lambda$. Unless specifically defined, an uppercase symbol represents the diagonal matrix with the items of the corresponding lowercase symbol representing a vector as diagonal elements appearing in the same order, i.e., $\fun{C}(x) = \operatorname{diag}\bigl(\fun{c}(x)\bigr)$ and $\Lambda = \operatorname{diag}(\lambda)$.

For a general function~$\fun{f}$, we denote the Jacobian of~$f$ by~$\fun{J_{\fun{f}}}$. Furthermore, for general functions $\fun{g}$~and~$\fun{h}$, we write $\fun{g}(\mu) = O\bigl(\fun{h}(\mu)\bigr)$ if there exists an $M > 0$ such that for all~$\mu$ with sufficiently small magnitude, $\lvert \fun{g}(\mu) \rvert \leq M \lvert \fun{h}(\mu) \rvert$. We write $\fun{g}(\mu) = \Omega\bigl(\fun{h}(\mu)\bigr)$ if $\fun{h}(\mu) = O\bigl(\fun{g}(\mu)\bigr)$ and write $\fun{g}(\mu) = \Theta\bigl(\fun{h}(\mu)\bigr)$ if $\fun{g}(\mu) = O\bigl(\fun{h}(\mu)\bigr)$ and $\fun{g}(\mu) = \Omega\bigl(h(\mu)\bigr)$. Using this notation, the dependency of the functions $g$~and~$h$ on~$\mu$ is sometimes only implied from the context.

For $x^*$ a solution to~\eqref{eqn:nlp}, we denote the set of indices~$i$ of the active constraints, for which $\fun{c_i}(x^*) = 0$, by~$\fun{\mathcal{A}}(x^*)$; the set of indices of inactive constraints is consequently given by $\{1, \ldots, m\} \setminus \fun{\mathcal{A}}(x^*)$, where we note that all equality constraints are active constraints. As abbreviated notation, we write $\fun{g}$ for the gradient of the objective function~$\fun{f}$, $\fun{H}$ for the Hessian with respect to~$x$ of the Lagrangian~$(x, \lambda) \mapsto \fun{f}(x) - \lambda^T \fun{c}(x)$ and $\fun{A}$ for the Jacobian of the vector-valued constraint function~$\fun{c}$, where a subscript applied to~$\fun{A}$ should be read as a subscript applied to~$\fun{c}$.

Finally, explicit references to multiples of the vector $\begin{pmatrix} 0 & e^T & 0 \end{pmatrix}^T$ should be interpreted with the understanding that the block components are of dimension $n$,~$m_{\mathcal{I}}$ and~$m_{\mathcal{E}}$ respectively.

\subsection{Interior-point methods}

Path-following primal--dual interior-point methods can be motivated by barrier methods, also called \emph{primal} interior-point methods; see, e.g., \cite{FM68,FGW02} for an extensive introduction to both. In a barrier method, inequality constraints are handled through the addition of a barrier term to the objective function that is scaled by~$\mu$, the barrier parameter, which in case of the (natural) log-barrier function results in the objective
\begin{equation*}
	\fun{B}(x, \mu) \triangleq \fun{f}(x) - \mu \sum_{i \in \mathcal{I}} \ln \fun{c_i}(x).
\end{equation*}
Under some conditions, for $\mu > 0$, the barrier function increases in an unbounded fashion for feasible points approaching the boundary, which can be exploited by iterative methods to implicitly enforce the constraint $\fun{c_\mathcal{I}}(x) > 0$. Now, the smaller~$\mu$, the better the barrier term approximates an indicator function for satisfying the inequality constraints strictly and the better the solution of the equality-constrained barrier problem approximates the solution of the original problem. The first-order necessary KKT optimality conditions for $x$ being a minimizer of the resulting problem with $[\lambda]_{\mathcal{E}}$ being the Lagrange multiplier vector to the equality constraints are
\begin{equation*}
	\begin{cases}
		0 = \grad_{\!x} \fun{B}(x, \mu) - \fun{A_\mathcal{E}}(x)^T [\lambda]_\mathcal{E}; \\
		0 = \fun{c_\mathcal{E}}(x),
	\end{cases}
\end{equation*}
where
\begin{equation*}
	\grad_{\!x} \fun{B}(x, \mu) = \grad \fun{f}(x) - \mu \sum_{i \in \mathcal{I}} \frac{1}{\fun{c_i}(x)} \grad \fun{c_i}(x) = \fun{g}(x) - \mu \fun{A_\mathcal{I}}(x)^T \fun{C_\mathcal{I}}(x)^{-1} e;
\end{equation*}
introducing $[\fun{\lambda}(x)]_\mathcal{I} \triangleq \mu \fun{C_\mathcal{I}}(x)^{-1} e$, these optimality conditions are equivalent to
\begin{equation*}
	\begin{cases}
		0 = \fun{g}(x) - \fun{A_\mathcal{I}}(x)^T [\fun{\lambda}(x)]_\mathcal{I} - \fun{A_\mathcal{E}}(x)^T [\lambda]_\mathcal{E}; \\
		 0 = \fun{C_\mathcal{I}}(x) [\fun{\lambda}(x)]_\mathcal{I} - \mu e \qquad \Leftrightarrow \qquad [\fun{\lambda}(x)]_\mathcal{I} = \mu \fun{C_\mathcal{I}}(x)^{-1} e; \\
		0 = \fun{c_\mathcal{E}}(x).
	\end{cases}
\end{equation*}
In a primal--dual interior-point method, the dependency of~$\fun{\lambda}(\cdot)$ on~$x$ is lifted and $[\lambda]_\mathcal{I}$ is treated as an independent variable, like~$[\lambda]_\mathcal{E}$. For a chosen barrier parameter, a solution~$(x, \lambda)$ under the implicit constraints~$\bigl(\fun{c_{\mathcal{I}}}(x), [\lambda]_{\mathcal{I}}\bigr) > 0$ to
\begin{equation*}
	0 = \fun{F^\mu}(x, \lambda) \triangleq \begin{pmatrix} \fun{g}(x) - \fun{A}(x)^T \lambda \\ \fun{C_\mathcal{I}}(x) [\lambda]_\mathcal{I} - \mu e \\ \fun{c_\mathcal{E}}(x) \end{pmatrix}
\end{equation*}
is sought, which are perturbed optimality conditions to the original problem~\eqref{eqn:nlp}, in the sense that the complementarity condition is perturbed by~$\mu$. As the Jacobian of~$\fun{F^{\mu}}$ is independent of the choice of~$\mu$, we drop the superscript when referring to it.

Rather than solving the perturbed system for a predetermined small value for the barrier parameter, which can be difficult to achieve efficiently, a common approach is to use \emph{outer} iterations to approximately solve perturbed problems for a decreasing sequence of barrier parameters in \emph{inner} iterations, where the next inner iteration is started using information about the solution of the previous. The hope here is that the solutions are close enough to each other, to limit the number of inner iterations needed: as shown in~\cite{FM68}, under some conditions, there exists a sufficiently smooth trajectory called the \emph{barrier trajectory} of solutions of the perturbed problem parameterized by the barrier parameter for a small enough values, including zero yielding the solution of the original problem, and hence the characterization of such methods as trajectory-following.

\subsection{Extrapolation methods}

It has been demonstrated in~\cite{FM68} how to use a Taylor-series approximation using analytical expressions for the derivatives of the barrier trajectory to obtain an approximation to the solution of the original problem at~$\mu = 0$ given the exact solution to the perturbed problem with~$\mu > 0$. More practically, an accelerator is described where the trajectory is approximated by a polynomial that goes through previously obtained approximate solutions for perturbed problems and this approximation is used to obtain a starting point for the next inner iteration as accelerator.

Following a different approach, the term extrapolation has been used in~\cite{BDM93} in the context of a primal penalty-barrier method, in which equality constraints are handled by penalizing the objective based on a measure of not attaining the constraints. At the start of each inner iteration, an extrapolation step is made by following a first-order Taylor-series approximation to an implicit function that describes both the current iteration point and the first-order optimal solution of the perturbed problem. Continuing the inner iteration with Newton steps until the solution of the perturbed problem is sufficiently well approximated, asymptotically, only a single Newton step is needed and two-step superlinear convergence was shown.

Similarly, for primal--dual interior-point methods, superlinear convergence has been proven in~\cite{GOST01} by taking a Newton step at the beginning of each inner iteration. An alternative view on this step is given as it being a combination of the step following the first-order Taylor-series approximation to an implicit function that keeps the residual of the perturbed optimality conditions but varies the barrier parameter and the Newton step using the barrier parameter of the previous inner iteration. One-step superlinear convergence for a modified version of the barrier method has been proven in~\cite{WJ99} in the case of a linear objective function by starting each inner iteration with a Newton step with the previous instead of current barrier parameter in the coefficient matrix.

A common approach is to solve linear systems in the Jacobian of the perturbed optimality conditions, of which constructing the matrix decomposition forms an expensive part. Ways of reusing it across different linear systems have been explored, of which Mehrotra's predictor--corrector method is an example of method that gained popularity -- introduced in~\cite{Meh91a} for linear programming but also widely used for solving quadratic programming problems. Each iteration, in which a single decomposition is used twice, consists of the combination of what is equivalent to a second-order Taylor-series approximation to the solution of the original problem -- computed in two linear systems -- and a first-order approximation to the barrier trajectory -- computed in the second linear system for the \emph{corrector} step based on the decrease in mean complementarity by following the possibly shortened \emph{predictor} step computed in the first system. For linear complementarity problems, an algorithm has been introduced in~\cite{WZ96} that uses a Shamanskii-like variant on Newton's method in which, after obtaining a Newton step by solving a linear system, systems using the same coefficient matrix but with updated right-hand sides by following the resulting steps get computed. By increasing the number of iterations, a theoretical arbitrary rate of convergence is obtained.
 
For methods using Taylor-series approximations to the solution of the perturbed problem, higher-order schemes have been given in \cite{Dus05}~and~\cite{Dus10} for primal penalty and barrier methods with asymptotic convergence rates and such a scheme has been proposed in~\cite{EV24} for a primal--dual interior-point method; for linear programming problems, convergence results have been given in~\cite{Car09} for methods using second-order approximations in the same setting. As noticed there, for linear complementarity problems, complexity bounds have been given in~\cite{ZZ95} for two algorithms following both the second-order approximation to the perturbed problem as well as the predictor--corrector spirit. A higher-order primal-dual interior-point method for quadratic programming problems has been analyzed in~\cite{EV22} that uses Taylor-series approximations to the solution of both the original problem and the perturbed problem.

In this work, the asymptotic behavior of a method using higher-order Taylor-series approximations to approximate the solution of the perturbed problem is studied for a primal--dual interior-point method. It provides a generalization of the results in the unpreconditioned case from~\cite{GOST01} to higher-order convergence rates at the cost of assuming an additional order of smoothness, with similar termination criteria for the inner iterations. Also, this work provides the missing convergence characteristic of such a method hypothesized in~\cite{Dus10} in the context of different interior-point methods. Comparing the steps taken in the proposed algorithm in~\cite{EV24} with the extrapolation step described in this work, this work provides local convergence theory for that algorithm.

In section~\ref{ch:trajectories}, the function to obtain the Taylor-series approximation to are formally defined, which are used in section~\ref{ch:extrapolation-step} to formally define the extrapolation step and obtain asymptotic properties of it. The needed computations, with an explicit description for the quadratic programming case, to obtain the extrapolation step are then described in section~\ref{ch:extrapolation-step-computation}, based on which the local convergence of a framework in which extrapolation steps are taken is described in section~\ref{ch:extrapolation-step-local-convergence}. Lastly, computational results for quadratic programming problems are shown in section~\ref{ch:numerical-experiments} to evaluate the performance of an extrapolation step as accelerator.

\section{Trajectories}
\label{ch:trajectories}

In this section, we define the functions that will be used in the next section as part of an extrapolation method, to approximate the solution of the perturbed problem with perturbation~$\mu$ using a Taylor-series approximation from any point at which the norm of~$\fun{F^\mu}$ is sufficiently small.

We start by stating our assumptions on the solution of~\eqref{eqn:nlp}.
\begin{assumption}
	\label{asm:licq}
	Given a KKT point~$x^* \in \mathbb{R}^n$ for the problem described by~\eqref{eqn:nlp}, assume the \emph{linear independence constraint qualification (LICQ)} holds at~$x^*$; that is, assume that the set~$\{\grad \fun{c_i}(x^*)\}_{i \in \mathcal{A}(x^*)}$ of active-constraint gradients at~$x^*$ consists of linearly independent vectors.
\end{assumption}
\begin{assumption}
	\label{asm:strict-complementarity}
	Given a KKT point~$x^* \in \mathbb{R}^n$ for the problem described by~\eqref{eqn:nlp}, assume that \emph{strict complementarity} holds at~$x^*$; that is, assume that there exists a~$\lambda^* \in \mathbb{R}^m$ that fulfills the conditions
	\begin{equation*}
		\fun{g}(x^*) = \fun{A}(x^*)^T \lambda^*, \quad \fun{C_\mathcal{I}}(x^*) [\lambda^*]_\mathcal{I} = 0 \quad \text{and} \quad [\lambda^*]_\mathcal{I} \geq 0
	\end{equation*}
	for being a Lagrange multiplier vector to the inequality constraints for which for all~$i \in \mathcal{I} \cap \fun{\mathcal{A}}(x^*)$, $[\lambda^*]_i > 0$.
\end{assumption}
\begin{assumption}
	\label{asm:strong-second-order-sufficiency-condition}
	Given a KKT point~$x^* \in \mathbb{R}^n$ for the problem described by~\eqref{eqn:nlp}, assume that the \emph{strong second-order sufficiency condition} is satisfied at~$x^*$; that is, assume that there exists an~$\omega > 0$ such that $p^T \fun{H}(x^*, \lambda^*) p \geq \omega \lVert p \rVert^2$ for all~$p \in \mathbb{R}^n$ for which for all~$i \in \fun{\mathcal{A}}(x^*)$, $\grad \fun{c_i}(x^*)^T p = 0$.
\end{assumption}

It follows under Assumption~\ref{asm:licq} that there exists for each KKT point~$x^*$ a unique Lagrange multiplier vector~$\lambda^*$ to~$\fun{c}$ at~$x^*$, which under Assumption~\ref{asm:strict-complementarity} has strictly positive components for the components corresponding to the active inequality constraints at~$x^*$.

The following result provides the basis for solving the problem using trajectories and is commonly used in different variations in the context of interior-point methods; see, e.g., \cite{FM68,BDM93,WJ99}.
\begin{lemma}
	\label{thm:existence-unique-barrier-trajectories}
	\begin{subequations} \label{eqn:implicit-function-general-trajectory}
	Let~$x^* \in \mathbb{R}^n$ be a KKT point for the problem described by~\eqref{eqn:nlp} under Assumptions~\ref{asm:licq}, \ref{asm:strict-complementarity} and \ref{asm:strong-second-order-sufficiency-condition}, such that there exists a unique Lagrange multiplier vector~$\lambda^*$ of problem~\eqref{eqn:nlp} to~$\fun{c}$ at~$x^*$. Then, there exists a locally unique function~$\fun{w^{w^*\!,0}} \colon \mathbb{R}^{(n + m)} \to \mathbb{R}^{(n + m)}$ depending on $w^* = (x^*, \lambda^*)$ that is $(d - 1)$~times continuously differentiable on a neighborhood of~$r = 0$ such that locally
	\begin{equation}
		\fun{F^0}\bigl(\fun{w^{w^*\!,0}}(r)\bigr) = r \label{eqn:implicit-function-general-trajectory-definition}
	\end{equation}
	and
	\begin{equation}
		\fun{w^{w^*\!,0}}(0) = w^*. \label{eqn:implicit-function-general-trajectory-limit}
	\end{equation}
	\end{subequations}
\end{lemma}
\begin{proof}
	Define
	\begin{equation*}
		\fun{h}(r, w) \triangleq \fun{F^0}(w) - r.
	\end{equation*}
	Clearly, $\fun{h}$ is as often differentiable as~$\fun{F^0}$ is: $(d - 1)$~times. Since $\fun{J_{\fun{F}}}(w^*)$ is invertible by~\cite[proof of Thm.~17]{FM68} under the stated assumptions and since
	\begin{equation*}
		\fun{h}(0, w^*) = \fun{F^0}(w^*) - 0 = 0,
	\end{equation*}
	the result follows from applying the implicit function theorem.
\end{proof}

While Lemma~\ref{thm:existence-unique-barrier-trajectories} guarantees the existence of a function through which an optimal solution to the problem can be found by~\eqref{eqn:implicit-function-general-trajectory-limit}, an analytical expression for it depends on this optimal solution, which is unknown, and what we are left with is the implicit definition~\eqref{eqn:implicit-function-general-trajectory-definition} only. However, differentiating this same \eqref{eqn:implicit-function-general-trajectory-definition} with respect to its argument, given the value of~$\fun{w^{w^*\!,0}}(r)$, we are able to obtain analytic expressions for the derivatives of the trajectory up to but not including the $d$th-order without explicit knowledge of the optimal point; this will later be explored in section~\ref{ch:extrapolation-step-computation}.

For the special case in~\eqref{eqn:implicit-function-general-trajectory-definition} of $r$ being a multiple of~$\begin{pmatrix} 0 & e^T & 0 \end{pmatrix}^T$, we define
\begin{equation*}
	\fun{w^{w^*}}(\mu) \triangleq \fun{w^{w^*\!,0}}\Bigl(\begin{pmatrix} 0 & \mu e^T & 0 \end{pmatrix}^T\Bigr),
\end{equation*}
which defines the barrier trajectory. Strict feasibility to the inequality constraints follows for the corresponding points for $\mu > 0$ from the assumption of strict complementarity, as will later be shown as part of the proof of Lemma~\ref{thm:primal-dual-higher-order-termination-criteria-fulfilling}.

Given that we are only interested in approximating~$\fun{w^{w^*}}(\mu)$ from a point~$\fun{w^{w^*\!,0}}(r)$, we consider a different function -- defined in a similar way to a function in~\cite{Dus05} -- in the following corollary. It joins those two points with a curve that is parameterized by a only single scalar, whose domain is chosen to scale with the distance between the points as in the original function. By using a single scalar argument, we are guided to the barrier trajectory with less degrees of freedom to handle when computing the Taylor-series approximation using the derivatives of the function.
\begin{corollary}
	\label{thm:existence-unique-barrier-trajectories-scalar}
	\begin{subequations} \label{eqn:implicit-function-barrier-scalar-trajectory}
	Let~$x^* \in \mathbb{R}^n$ be a KKT point for the problem described by~\eqref{eqn:nlp} under Assumptions \ref{asm:licq},~\ref{asm:strict-complementarity} and~\ref{asm:strong-second-order-sufficiency-condition}, such that there exists a unique Lagrange multiplier vector~$\lambda^*$ of problem~\eqref{eqn:nlp} to~$\fun{c}$ at~$x^*$. For any real-valued vector~$r$, we define the function~$\operatorname{nml}$ to normalize~$r$ under the relation $\operatorname{nml}(r) \lVert r \rVert \equiv r$ through
	\begin{equation*}
		\operatorname{nml}(r) \triangleq \begin{cases} \frac{r}{\lVert r \rVert}, & r \neq 0; \\ 0, & \text{otherwise}. \end{cases}
	\end{equation*}
	Then, there exists a function~$\fun{w^{w^*\!, \mu, r}} \colon \mathbb{R} \to \mathbb{R}^{(n + m)}$ depending on $w^* = (x^*, \lambda^*)$ for all $\mu$~and~$r$ independently sufficiently small that is $(d - 1)$~times continuously differentiable on a neighborhood of~$\rho \in [0, \lVert r \rVert ]$ such that locally
	\begin{equation}
		\fun{F^\mu}\bigl(\fun{w^{w^*\!, \mu, r}}(\rho)\bigr) = \rho \operatorname{nml}(r) \label{eqn:implicit-function-barrier-scalar-trajectory-definition}
	\end{equation}
	and
	\begin{equation}
		\fun{w^{w^*\!, \mu, r}}(0) = \fun{w^{w^*}}(\mu). \label{eqn:implicit-function-barrier-scalar-trajectory-limit}
	\end{equation}
	\end{subequations}
\end{corollary}
\begin{proof}
	For all $\mu$~and~$r$ independently sufficiently small, $r + \begin{pmatrix} 0 & \mu e^T & 0 \end{pmatrix}^T$ lies in the neighborhood of~$0$ on which Lemma~\ref{thm:existence-unique-barrier-trajectories} guarantees the existence of a $(d - 1)$~times continuously differentiable function $\fun{w^{w^*}} \colon \mathbb{R}^{(n + m)} \to \mathbb{R}^{(n + m)}$ depending on~$w^*$ that locally fulfills~\eqref{eqn:implicit-function-general-trajectory}. We can then define the equally smooth function $\fun{w^{w^*\!, \mu}} \colon \mathbb{R}^{(n + m)} \to \mathbb{R}^{(n + m)}$ for those small values of~$r$ by
	\begin{equation*}
		\fun{w^{w^*\!, \mu}}(r) \triangleq \fun{w^{w^*}}\Bigl(r + \begin{pmatrix} 0 & \mu e^T & 0 \end{pmatrix}^T\Bigr).
	\end{equation*}
	With this, locally
	\begin{align*}
		\fun{F^\mu}\bigl(\fun{w^{w^*\!, \mu}}(r)\bigr)
		&= \fun{F^0}\bigl(\fun{w^{w^*\!, \mu}}(r)\bigr) - \begin{pmatrix} 0 & \mu e^T & 0 \end{pmatrix}^T \\
		&= \fun{F^0}\Biggl(\fun{w^{w^*}}\Bigl(r + \begin{pmatrix} 0 & \mu e^T & 0 \end{pmatrix}^T\Bigr)\Biggr) - \begin{pmatrix} 0 & \mu e^T & 0 \end{pmatrix}^T \\
		&= r + \begin{pmatrix} 0 & \mu e^T & 0 \end{pmatrix}^T - \begin{pmatrix} 0 & \mu e^T & 0 \end{pmatrix}^T = r
	\end{align*}
	and
	\begin{align*}
		\fun{w^{w^*\!, \mu}}(0) = \fun{w^{w^*}}(\mu).
	\end{align*}
	Moreover, it follows directly from the above that the function defined by
	\begin{equation*}
		\fun{w^{w^*\!, \mu, r}}(\rho) \triangleq \fun{w^{w^*\!, \mu}}\bigl(\rho \operatorname{nml}(r)\bigr)
	\end{equation*}
	fulfills the desired properties.
\end{proof}

A reason for going through the function defined in Lemma~\ref{thm:existence-unique-barrier-trajectories} instead of deriving~$\fun{w^{w^*\!, \mu}}$ directly using the implicit function theorem, as done in the proof there, is to let the notion of sufficiently small for~$r$ be independent from~$\mu$, as will be used in the next section for an extrapolation method.

A natural question to ask is what happens outside of the neighborhood provided by Lemma~\ref{thm:existence-unique-barrier-trajectories}. There, solutions in~$w$ to $\fun{F^0}(w) = r$ are not necessarily unique and form a continuous trajectory: see \cite{CGGK05} for examples. Outside of this neighborhood, we therefore cannot obtain an extrapolation step with the interpretation of it being a step to the Taylor-series approximation of the above function, but, as will be described in section~\ref{ch:extrapolation-step-computation}, it is possible to describe a step that equals the extrapolation step in the neighborhood for all points~$w$ for which $\fun{J_{\fun{F}}}(w)$ is nonsingular.

\section{Extrapolation step}
\label{ch:extrapolation-step}

In this section, we consider the asymptotic behavior and the speed of convergence of methods based on extrapolation of the previously described function to the barrier trajectory. As we will see, by taking an extrapolation step as first step after decreasing the barrier parameter, asymptotically, the stopping criteria for the inner minimization method will immediately be satisfied.

For the starting point~$w_{k+1}$ of outer iteration~$k + 1$, to fit the notation of the function defined previously, we assign a name to the residual through $r_{k+1} \triangleq \fun{F^{\mu_{k+1}}}(w_{k+1})$. Using this, we define~$w^{w^*\!, p}_{k+1}$ for $p < d - 1$ as the $p$th-order Taylor-series approximation at
\begin{equation}
	\fun{w^{w^*\!, \mu_{k+1}, r_{k+1}}}(\lVert r_{k+1} \rVert) \label{eqn:primal-dual-extrapolation-iteration}
\end{equation}
to
\begin{equation*}
	\fun{w^{w^*}}(\mu_{k+1}) = \fun{w^{w^*\!, \mu_{k+1}, r_{k+1}}}(0)
\end{equation*}
for those points~$w_{k+1}$ for which this concept is well defined in accordance with Corollary~\ref{thm:existence-unique-barrier-trajectories-scalar}. The componentwise error of this approximation is by Taylor's theorem for all $j = 1, \ldots, n + m$,
\begin{equation}
	\bigl[\fun{w^{w^*}}(\mu_{k+1}) - w^{w^*\!, p}_{k+1}\bigr]_j = \fun{O}\bigl(\lVert r_{k+1} \rVert^{p+1}\bigr), \label{eqn:primal-dual-extrapolation-error}
\end{equation}
where the use of the $\fun{O}$~notation is justified by the $(d - 1)$~times continuously differentiability of the function involved.

In the context of the following lemma describing asymptotic properties of the extrapolation step, similar to the setting in~\cite{GOST01}, we use for the inner minimization the termination criterion
\begin{equation}
	\lVert \fun{F^{\mu_k}}(w_{k+1}) \rVert \leq \fun{\epsilon}(\mu_k). \label{eqn:primal-dual-termination} 
\end{equation}
for~$\fun{\epsilon}$ a positive scalar function such that $\fun{\epsilon}(\mu_k) = \fun{\Theta}(\mu_k)$. Notably, the implicit constraints $(\fun{c_\mathcal{I}}(x_{k+1}), [\lambda_{k+1}]_\mathcal{I}) > 0$ that are part of the perturbed problem are missing here. In the presented analysis, we assume the existence of a subsequence of iterates converging to a solution to the original problem by staying in a neighborhood of the barrier trajectory and the requirement of strict feasibility is therefore implied by the assumption of strict complementarity.
\begin{lemma}
	\label{thm:primal-dual-higher-order-termination-criteria-fulfilling}
	Under the assumptions of Lemma~\ref{thm:existence-unique-barrier-trajectories-scalar}, including Assumption~\ref{asm:strict-complementarity}, let $\{\mu_k\}_{k \in \mathbb{N}}$ be a strictly decreasing sequence of positive scalars and let $\{w_k\}_{k \in \mathbb{N}}$ be a sequence of iterates fulfilling~\eqref{eqn:primal-dual-termination} such that there exists a subsequence indexed by~$\mathcal{K}$ for which $\{w_{k+1}\}_{k \in \mathcal{K}} \to w^*$. Furthermore, let $p \in \{1, \ldots, d - 2\}$.
	Then, for all $k \in \mathcal{K}$ sufficiently large, $w^{w^*\!, p}_{k+1}$ is well defined and $w_{k+1}$ equals the expression in~\eqref{eqn:primal-dual-extrapolation-iteration}. Assuming $\mu_{k+1} = \fun{\Omega}\bigl(\mu_k^{p+\gamma}\bigr)$ for $\gamma \in (0, 1)$, then
	\begin{subequations}
	\label{eqn:primal-dual-extrapolation-termination}
	\begin{align}
		\bigl\lVert \fun{F^{\mu_{k+1}}}(w^{w^*\!, p}_{k+1}) \bigr\rVert &\leq \fun{\epsilon}(\mu_{k+1}) \quad \text{and} \label{eqn:primal-dual-extrapolation-termination-norm} \\
		\bigl(\fun{c_\mathcal{I}}\bigl(x^{w^*\!, p}_{k+1}\bigr), [\lambda^{w^*\!, p}_{k+1}]_\mathcal{I}\bigr) &> 0. \label{eqn:primal-dual-extrapolation-termination-positive}
	\end{align}
	\end{subequations}
	Also, for all~$j = 1, \ldots, n + m$,
	\begin{equation}
		\bigl[w^{w^*\!, p}_{k+1} - w^*\bigr]_j = \fun{O}(\mu_{k+1}) \label{eqn:primal-dual-extrapolation-distance-optimality-upper-bound}
	\end{equation}
	and more specifically, for those values of~$j$ for which $\bigl[\fun{\dot{w}^{w^*}}(0)\bigr]_j \neq 0$,
	\begin{equation}
		\bigl[w^{w^*\!, p}_{k+1} - w^*\bigr]_j = \fun{\Theta}(\mu_{k+1}). \label{eqn:primal-dual-extrapolation-distance-optimality}
	\end{equation}
\end{lemma}
\begin{proof}
	Applying the triangle inequality, we write
	\begin{equation*}
		\lVert r_{k+1} \rVert = \lVert \fun{F^{\mu_{k+1}}}(w_{k+1}) \rVert \leq \lVert \fun{F^{\mu_k}}(w_{k+1}) \rVert + (\mu_k - \mu_{k+1}) \Bigl\lVert \begin{pmatrix} 0 & e^T & 0 \end{pmatrix}^T \Bigr\rVert = \fun{O}(\mu_k),
	\end{equation*}
	where the final equality is by~\eqref{eqn:primal-dual-termination} and the decreaseness and positivity of the sequence~$\{\mu_k\}_{k \in \mathbb{N}}$ of barrier parameters which implies $\mu_k - \mu_{k+1} = \fun{O}(\mu_k)$. Now that $\lVert r_{k+1} \rVert = \fun{O}(\mu_k)$, it follows that, for $k \in \mathcal{K}$ sufficiently large, $\mu_{k+1}$~and~$r_{k+1}$ are sufficiently small such that Corollary~\ref{thm:existence-unique-barrier-trajectories-scalar} provides a unique $(d - 1)$~times continuously differentiable function~$\fun{w^{w^*\!, \mu_{k+1}, r_{k+1}}}$ that satisfies~\eqref{eqn:implicit-function-barrier-scalar-trajectory}, such that $w^{w^*\!, p}_{k+1}$ is well defined. As
	\begin{equation*}
		\fun{F^{\mu_{k+1}}}(w_{k+1}) = r_{k+1} = \fun{F^{\mu_{k+1}}}\bigl(\fun{w^{w^*\!, \mu_{k+1}, r_{k+1}}}(\lVert r_{k+1} \rVert)\bigl)
	\end{equation*}
	and
	\begin{equation*}
		\lim_{k \in \mathcal{K} \to \infty} w_{k+1} = w^* = \lim_{k \in \mathcal{K} \to \infty} \fun{w^{w^*\!, \mu_{k+1}, r_{k+1}}}(\lVert r_{k+1} \rVert),
	\end{equation*}
	it follows from the uniqueness that, for $k \in \mathcal{K}$ sufficiently large, $w_{k+1}$ equals the expression in~\eqref{eqn:primal-dual-extrapolation-iteration}. Also using the relative magnitude of~$r_{k+1}$, \eqref{eqn:primal-dual-extrapolation-error} gives us that for all $j = 1, \ldots, n + m$,
	\begin{equation*}
		\bigl[\fun{w^{w^*}}(\mu_{k+1}) - w^{w^*\!, p}_{k+1}\bigr]_j = \fun{O}\bigl(\lVert r_{k+1} \rVert^{p+1}\bigr) = \fun{O}\bigl(\mu_k^{p+1}\bigr).
	\end{equation*}
	Moreover, since
	\begin{equation*}
		\mu_{k+1} = \fun{\Omega}\bigl(\mu_k^{p+\gamma}\bigr) \quad \Leftrightarrow \quad \mu_k^{p+\gamma} = \fun{O}(\mu_{k+1}) \quad \Leftrightarrow \quad \mu_k^{p+1} = \fun{O}\Bigl(\mu_{k+1}^{\frac{p+1}{p+\gamma}}\Bigr),
	\end{equation*}
	it follows, flipping the sign, that
	\begin{equation*}
		\bigl[w^{w^*\!, p}_{k+1} - \fun{w^{w^*}}(\mu_{k+1})\bigr]_j = \fun{O}\Bigl(\mu_{k+1}^\frac{p+1}{p+\gamma}\Bigr)
	\end{equation*}
	and since $p + 1 > p + \gamma$, we get that $\frac{p+1}{p+\gamma} > 1$, i.e., that the exponent is bigger than~$1$.

	Applying Taylor's theorem componentwise, we see that for all $j = 1, \ldots, n + m$,
	\begin{equation*}
		\bigl[\fun{F^{\mu_{k+1}}}\bigl(w^{w^*\!, p}_{k+1}\bigr)\bigl]_j = \bigl[\fun{F^{\mu_{k+1}}}\bigl(\fun{w^{w^*}}(\mu_{k+1})\bigr)\bigr]_j + \fun{O}\bigl(\bigl\lVert w^{w^*\!, p}_{k+1} - \fun{w^{w^*}}(\mu_{k+1}) \bigr\rVert\bigr) = \fun{O}\Bigl(\mu_{k+1}^\frac{p+1}{p+\gamma}\Bigr),
	\end{equation*}
	where the last equality is because $\fun{F^{\mu_{k+1}}}\bigl(\fun{w^{w^*}}(\mu_{k+1})\bigr) = 0$. This shows~\eqref{eqn:primal-dual-extrapolation-termination-norm}, since $\fun{\epsilon}(\mu_k) = \fun{\Omega}(\mu_k)$.

	We will now prove~\eqref{eqn:primal-dual-extrapolation-termination-positive}. Using Taylor's theorem, for all~$i \in \mathcal{I}$,
	\begin{equation*}
		\fun{c_i}\bigl(x^{w^*\!, p}_{k+1}\bigr) = \fun{c_i}\bigl(\fun{x^{w^*}}(\mu_{k+1})\bigr) + \fun{O}\bigl(\bigl\lVert x^{w^*\!, p}_{k+1} - \fun{x^{w^*}}(\mu_{k+1})\bigr\rVert\bigr) = \fun{c_i}\bigl(\fun{x^{w^*}}(\mu_{k+1})\bigr) + \fun{O}\Bigl(\mu_{k+1}^\frac{p+1}{p+\gamma}\Bigr),
	\end{equation*}
	as $c_i$ is continuously differentiable, and also
	\begin{equation*}
		\bigl[\lambda^{w^*\!, p}_{k+1}\bigr]_i = \bigl[\fun{\lambda^{w^*}}(\mu_{k+1})\bigr]_i + \fun{O}\bigl(\bigl\lVert \lambda^{w^*\!, p}_{k+1} - \fun{\lambda^{w^*}}(\mu_{k+1}) \bigr\rVert\bigr) = \bigl[\fun{\lambda^{w^*}}(\mu_{k+1})\bigr]_i + \fun{O}\Bigl(\mu_{k+1}^\frac{p+1}{p+\gamma}\Bigr).
	\end{equation*}
	We will here distinguish between the case for active and for inactive inequality constraints. First, let $i \in \mathcal{I} \cap \fun{\mathcal{A}}(x^*)$ be the index of an inequality constraint that is active at~$x^*$. By strict complementarity, $\bigl[\fun{\lambda^{w^*}}(0)\bigr]_i > 0$ and by a continuity argument, $\bigl[\fun{\lambda^{w^*}}(\mu_{k+1})\bigr]_i = \fun{\Theta}(1)$; since $\fun{c_i}\bigl(\fun{x^{w^*}}(\mu_{k+1})\bigr) \bigl[\fun{\lambda^{w^*}}(\mu_{k+1})\bigr]_i = \mu_{k+1}$, also $\fun{c_i}\bigl(\fun{x^{w^*}}(\mu_{k+1})\bigr) = \fun{\Theta}(\mu_{k+1})$. Now, let $i \in \mathcal{I} \setminus \fun{\mathcal{A}}(x^*)$; using the same reasoning, as $\fun{c_i}\bigl(\fun{x^{w^*}}(0)\bigr) > 0$, in this case $\fun{c_i}\bigl(\fun{x^{w^*}}(\mu_{k+1})\bigr) = \fun{\Theta}(1)$ and $\bigl[\fun{\lambda^{w^*}}(\mu_{k+1})\bigr]_i = \fun{\Theta}(\mu_{k+1})$. What is common between those cases, is that both $\fun{c_i}\bigl(\fun{x^{w^*}}(\mu_{k+1})\bigr)$~and~$\bigl[\fun{\lambda^{w^*}}(\mu_{k+1})\bigr]_i$ are strictly positive for all~$k \in \mathcal{K}$ sufficiently large and bounded below by a multiple of~$\mu_{k+1}$ with some exponent that is strictly smaller than that of the upper bound of its perturbation in the previous expression for $\fun{c_i}\bigl(x^{w^*\!, p}_{k+1}\bigr)$~and~$\bigl[\lambda^{w^*\!, p}_{k+1}\bigr]_i$. With that, we can asymptotically disregard the perturbation and conclude that those values are strictly positive too for all~$k \in \mathcal{K}$ sufficiently large, which concludes the the proof of~\eqref{eqn:primal-dual-extrapolation-termination-positive}.

	Lastly, we prove \eqref{eqn:primal-dual-extrapolation-distance-optimality-upper-bound} and \eqref{eqn:primal-dual-extrapolation-distance-optimality}. By Taylor's theorem, for all $j = 1, \ldots, n + m$,
	\begin{equation*}
		\bigl[\fun{w^{w^*}}(\mu_{k+1})\bigr]_j = \bigl[\fun{w^{w^*}}(0)\bigr]_j + \mu_{k+1} \bigl[\fun{\dot{w}^{w^*}}(0)\bigr]_j + \fun{O}(\mu_{k+1}^2),
	\end{equation*}
	from which it follows that $\bigl[\fun{w^{w^*}}(\mu_{k+1}) - \fun{w^{w^*}}(0)\bigr]_j = \fun{O}(\mu_{k+1})$ and for all~$j$ such that $\bigl[\fun{\dot{w}^{w^*}}(0)\bigr]_j \neq 0$, $\bigl[\fun{w^{w^*}}(\mu_{k+1}) - \fun{w^{w^*}}(0)\bigr]_j = \fun{\Theta}(\mu_{k+1})$. Using this, writing~$w^*$ as~$\fun{w^{w^*}}(0)$, we can see that for all~$j$,
	\begin{align*}
		\bigl[w^{w^*\!, p}_{k+1} - w^*\bigr]_j &= \bigl[w^{w^*\!, p}_{k+1} - \fun{w^{w^*}}(\mu_{k+1}) + \fun{w^{w^*}}(\mu_{k+1}) - \fun{w^{w^*}}(0)\bigr]_j \\
		&= \bigl[w^{w^*\!, p}_{k+1} - \fun{w^{w^*}}(\mu_{k+1})\bigr]_j + \bigl[\fun{w^{w^*}}(\mu_{k+1}) - \fun{w^{w^*}}(0)\bigr]_j \\
		&= \fun{O}\Bigl(\mu_{k+1}^\frac{p+1}{p+\gamma}\Bigr) + \fun{O}(\mu_{k+1}) = \fun{O}(\mu_{k+1}),
	\end{align*}
	and, repeating the argument, for all those~$j$ such that $\bigl[\fun{\dot{w}^{w^*}}(0)\bigr]_j \neq 0$,
	\begin{equation*}
		\bigl[w^{w^*\!, p}_{k+1} - w^*\bigr]_j = \fun{O}\Bigl(\mu_{k+1}^\frac{p+1}{p+\gamma}\Bigr) + \fun{\Theta}(\mu_{k+1}) = \fun{\Theta}(\mu_{k+1}),
	\end{equation*}
	which concludes the proof.
\end{proof}

\section{Computation of extrapolation step}
\label{ch:extrapolation-step-computation}

Having seen the effect of taking the extrapolation step on the minimization problem, this section concerns the computation of the step.

By the definition of the extrapolation step as Taylor-series approximation to~$\rho = 0$, introducing
\begin{equation*}
	\hat{w}_{k+1}^{w^*\!,q} \triangleq \frac{\diff^q \fun{w^{w^*\!,\mu_{k+1},r_{k+1}}}(\rho)}{\diff \rho^q}\biggr\rvert_{\rho = \lVert r_{k+1} \rVert} \cdot (0 - \lVert r_{k+1} \rVert)^q,
\end{equation*}
the step is given by $w^{w^*\!, p}_{k+1} = \sum_{q = 0}^p \frac{1}{q!} \hat{w}^{w^*\!, q}_{k+1}$, which is defined in terms of the derivatives of~$w^{w^*\!,\mu_{k+1},r_{k+1}}$. Differentiating the equivalence~\eqref{eqn:implicit-function-barrier-scalar-trajectory-definition} with respect to~$\rho$, we obtain
\begin{equation}
	\fun{J_{\fun{F}}}\bigl(w^{w^*\!,\mu,r}(\rho)\bigr) \frac{\diff \fun{w^{w^*\!,\mu,r}}(\rho)}{\diff \rho} \\
	= \operatorname{nml}(r),
	\label{eqn:implicit-definition-derivative}
\end{equation}
which allows us to obtain an expression for~$w^{w^*\!, p}_{k+1}$ in the case of $p = 1$.
\begin{proposition}
	\label{thm:w-1-newton}
	Under the assumptions of Lemma~\ref{thm:existence-unique-barrier-trajectories-scalar},including Assumption~\ref{asm:strict-complementarity}, let $\{\mu_k\}_{k \in \mathbb{N}}$ be a strictly decreasing sequence of positive scalars and let $\{w_k\}_{k \in \mathbb{N}}$ be a sequence of iterates fulfilling~\eqref{eqn:primal-dual-termination} such that there exists a subsequence indexed by~$\mathcal{K}$ for which $\{w_{k+1}\}_{k \in \mathcal{K}} \to w^*$.
	Then, for all $k \in \mathcal{K}$ sufficiently large,
	\begin{equation}
		w^{w^*\!, 1}_{k+1} = w_{k+1} - \fun{J_{\fun{F}}}(w_{k+1})^{-1} \fun{F^{\mu_{k+1}}}(w_{k+1}) \label{eqn:w-1-newton}
	\end{equation}
	and $w^{w^*\!, 1}_{k+1} - w_{k+1}$ is the Newton step for finding a root of~$\fun{F^{\mu_{k+1}}}$ at~$w_{k+1}$.
\end{proposition}
\begin{proof}
	By Lemma~\ref{thm:primal-dual-higher-order-termination-criteria-fulfilling}, for $k \in \mathcal{K}$ sufficiently large, $w^{w^*\!, p}_{k+1}$ is well defined and $w_{k+1} = \fun{w^{w^*\!, \mu_{k+1}, r_{k+1}}}(\lVert r_{k+1} \rVert)$. Writing out the expression obtained by definition of~$w^{w^*\!, 1}_{k+1}$ as first-order Taylor-series approximation and using~\eqref{eqn:implicit-definition-derivative}, we get
	\begin{align*}
		w^{w^*\!, 1}_{k+1} &= \fun{w^{w^*\!, \mu_{k+1}, r_{k+1}}}(\lVert r_{k+1} \rVert) + \frac{\diff \fun{w^{w^*\!,\mu_{k+1},r_{k+1}}}(\rho)}{\diff \rho}\biggr\rvert_{\rho = \lVert r_{k+1} \rVert} \cdot (0 - \lVert r_{k+1} \rVert) \\
		&= \fun{w^{w^*\!, \mu_{k+1}, r_{k+1}}}(\lVert r_{k+1} \rVert) - \fun{J_{\fun{F}}}\bigl(w^{w^*\!,\mu_{k+1},r_{k+1}}(\rho)\bigr)^{-1} \operatorname{nml}(r_{k+1}) \cdot \lVert r_{k+1} \rVert \\
		&= w_{k+1} - \fun{J_{\fun{F}}}(w_{k+1})^{-1} r_{k+1} \\
		&= w_{k+1} - \fun{J_{\fun{F}}}(w_{k+1})^{-1} \fun{F^{\mu_{k+1}}}(w_{k+1}),
	\end{align*}
	as desired.
\end{proof}

For affine equality constraints, the mechanism of satisfying those after a Newton step is also present for the extrapolation step, as demonstrated by the following proposition.
\begin{proposition}
	\label{thm:w-p-affine-equality}
	Under the assumptions of Lemma~\ref{thm:existence-unique-barrier-trajectories-scalar}, including Assumption~\ref{asm:strict-complementarity}, let $\{\mu_k\}_{k \in \mathbb{N}}$ be a strictly decreasing sequence of positive scalars and let $\{w_k\}_{k \in \mathbb{N}}$ be a sequence of iterates fulfilling~\eqref{eqn:primal-dual-termination} such that there exists a subsequence indexed by~$\mathcal{K}$ for which $\{w_{k+1}\}_{k \in \mathcal{K}} \to w^*$. Furthermore, let $p \in \{1, \ldots, d - 2\}$. Let~$\mathcal{E}_{\mathrm{A}} \subseteq \mathcal{E}$ such that there exists an $A_{\mathcal{E}_{\mathrm{A}}} \in \mathbb{R}^{\lvert \mathcal{E}_{\mathrm{A}} \rvert \times n}$ such that $A_{\mathcal{E}_{\mathrm{A}}} \equiv \fun{A_{\mathcal{E}_{\mathrm{A}}}}(x)$, i.e,., that \eqref{eqn:nlp} describes a problem with the constraints indexed by~$\mathcal{E}_{\mathrm{A}}$ being affine equality constraints. Then, for all $k \in \mathcal{K}$ sufficiently large,
	\begin{equation*}
		\fun{c_{\mathcal{E}_{\mathrm{A}}}}\bigl(x^{w^*\!, p}_{k+1}\bigr) = 0.
	\end{equation*}
\end{proposition}
\begin{proof}
	By~\eqref{eqn:implicit-definition-derivative}, $A_{\mathcal{E}_{\mathrm{A}}} \frac{\diff \fun{x^{w^*\!,\mu,r}}(\rho)}{\diff \rho}$ is equivalent to an expression constant in~$\rho$ and thus, for all~$q \geq 2$, $A_{\mathcal{E}_{\mathrm{A}}} \frac{\diff^q \fun{x^{w^*\!,\mu,r}}(\rho)}{\diff \rho^q} \equiv 0$. Therefore, $A_{\mathcal{E}_{\mathrm{A}}} x^{w^*\!, p}_{k+1} = A_{\mathcal{E}_{\mathrm{A}}} x^{w^*\!, 1}_{k+1}$ and as the first-order Taylor-series approximation of an affine function is perfect,
	\begin{align*}
		\fun{c_{\mathcal{E}_{\mathrm{A}}}}\bigl(x^{w^*\!, p}_{k+1}\bigr) &= \fun{c_{\mathcal{E}_{\mathrm{A}}}}(x_{k+1}) + A_{\mathcal{E}_{\mathrm{A}}} \bigl(x^{w^*\!, p}_{k+1} - x_{k+1}\bigr) = \fun{c_{\mathcal{E}_{\mathrm{A}}}}(x_{k+1}) + A_{\mathcal{E}_{\mathrm{A}}} \bigl(x^{w^*\!, 1}_{k+1} - x_{k+1}\bigr) \\
		&= \fun{c_{\mathcal{E}_{\mathrm{A}}}}(x_{k+1}) - \fun{c_{\mathcal{E}_{\mathrm{A}}}}(x_{k+1}) = 0,
	\end{align*}
	as desired.
\end{proof}

It can be observed that the expression for~$w^{w^*\!, 1}_{k+1}$ obtained in~\eqref{eqn:w-1-newton} does not actually depend on~$w^*$ and that the expression can be evaluated for all~$k$ and not only for $k \in \mathcal{K}$ sufficiently large -- as long as $\fun{J_{\fun{F}}}(w_{k+1})$ is invertible. Consequently, we can define $w^1_{k+1}$ through
\begin{equation*}
	w^1_{k+1} \triangleq w_{k+1} - \fun{J_{\fun{F}}}(w_{k+1})^{-1} \fun{F^{\mu_{k+1}}}(w_{k+1}),
\end{equation*}
an expression that can be evaluated if $\fun{J_{\fun{F}}}(w_{k+1})$ is invertible and that equals~$w^{w^*\!, 1}_{k+1}$ under the assumptions of Proposition~\ref{thm:w-1-newton} for $k \in \mathcal{K}$ sufficiently large. In fact, such generalization of~$w^{w^*\!, p}_{k+1}$ can be obtained for all orders of extrapolation~$p$: \eqref{eqn:implicit-function-barrier-scalar-trajectory-definition} used to obtain the derivatives does not depend on~$w^*$ and the unknown function~$\fun{w^{w^*\!, \mu_{k+1}, r_{k+1}}}$ is only evaluated at~$(\lVert r_{k+1} \rVert)$, for which the function value can be replaced by~$w_{k+1}$ by Lemma~\ref{thm:primal-dual-higher-order-termination-criteria-fulfilling}. Similarly, we define $w^p_{k+1}$ to be equal to the expression for $w^{w^*\!, p}_{k+1}$ with no other references to~$\fun{w^{w^*\!, \mu_{k+1}, r_{k+1}}}$ present than those through~$\fun{w^{w^*\!, \mu_{k+1}, r_{k+1}}}(\lVert r_{k+1} \rVert)$, and with this expression replaced by~$w_{k+1}$ -- for an expression that is independent of~$w^*$; also in this case, $w^p_{k+1}$ exists only exactly if $\fun{J_{\fun{F}}}(w_{k+1})$ is invertible, as is the case for $k \in \mathcal{K}$ sufficiently large. Also, the terms of the Taylor-series approximation for successive values of~$q$ can be computed as the solution of a linear system with the same coefficient matrix~$\fun{J_{\fun{F}}}(w_{k+1})$, but with different right-hand sides, of in general increasing complexity.

As an example, we will derive the necessary formulas for computing the extrapolation step in case of a quadratic programming problem.
\begin{proposition}
	Assume that there exist $H \in \mathbb{R}^{n \times n}$, $A_\mathcal{I} \in \mathbb{R}^{m_\mathcal{I} \times n}$ and $A_\mathcal{E} \in \mathbb{R}^{m_\mathcal{E} \times n}$ such that $H \equiv \fun{H}(x, \lambda)$, $A_\mathcal{I} \equiv \fun{A_\mathcal{I}}(x)$ and $A_\mathcal{E} \equiv \fun{A_\mathcal{E}}(x)$, i.e., that \eqref{eqn:nlp} describes a problem with a quadratic objective function and affine inequality and equality constraints. Then, for all $q \geq 1$,
	\begin{equation}
		\fun{J_{\fun{F}}}\bigl(w_{k+1}\bigr) \hat{w}_{k+1}^{w^*\!, q+1} = - \sum_{i = 1}^q \binom{q + 1}{i} \begin{pmatrix} 0 \\ \bigl[\hat{\Lambda}_{k+1}^{w^*\!,i}\bigr]_{\mathcal{I}} A_{\mathcal{I}} \hat{x}_{k+1}^{w^*\!,i} \\ 0 \end{pmatrix}.
		\label{eqn:w-hat-qp}
	\end{equation}
	\label{thm:w-hat-qp}
\end{proposition}
\begin{proof}
	Writing out the block rows of~\eqref{eqn:implicit-definition-derivative}, we can see that the constant~$\operatorname{nml}(r)$ equals
	\begin{equation*}
		\begin{pmatrix}
			H \frac{\diff \fun{x^{w^*\!,\mu,r}}(\rho)}{\diff \rho} - A \frac{\diff \fun{\lambda^{w^*\!,\mu,r}}(\rho)}{\diff \rho} \\
			\bigl[\fun{\Lambda^{w^*\!,\mu,r}}(\rho)\bigr]_{\mathcal{I}} A_{\mathcal{I}} \frac{\diff \fun{x^{w^*\!,\mu,r}}(\rho)}{\diff \rho} + \frac{\diff \bigl[\fun{\Lambda^{w^*\!,\mu,r}}(\rho)\bigr]_{\mathcal{I}}}{\diff \rho} A_{\mathcal{I}} \fun{x^{w^*\!,\mu,r}}(\rho) + \frac{\diff \bigl[\fun{\Lambda^{w^*\!,\mu,r}}(\rho)\bigr]_{\mathcal{I}}}{\diff \rho} \fun{c_{\mathcal{I}}}(0) \\
			A_\mathcal{E} \frac{\diff \fun{x^{w^*\!,\mu,r}}(\rho)}{\diff \rho}
		\end{pmatrix},
	\end{equation*}
	for which it is used that the first-order Taylor-series approximation for affine functions is exact and that the the role of the two vectors in the product of a diagonalized vector and a vector can be switched through
	\begin{align}
		\begin{split}
		\fun{C_{\mathcal{I}}}(\fun{x^{w^*\!,\mu,r}}(\rho)) \frac{\diff \bigl[\fun{\lambda^{w^*\!,\mu,r}}(\rho)\bigr]_{\mathcal{I}}}{\diff \rho} &= \frac{\diff \bigl[\fun{\Lambda^{w^*\!,\mu,r}}(\rho)\bigr]_{\mathcal{I}}}{\diff \rho} \fun{c_{\mathcal{I}}}(\fun{x^{w^*\!,\mu,r}}(\rho)) \\
		&= \frac{\diff \bigl[\fun{\Lambda^{w^*\!,\mu,r}}(\rho)\bigr]_{\mathcal{I}}}{\diff \rho} \bigl(\fun{c_{\mathcal{I}}}(0) + A_{\mathcal{I}} \fun{x^{w^*\!,\mu,r}}(\rho)\bigr).
		\end{split} \label{eqn:implicit-definition-derivative-c-lambda-rewrite}
	\end{align}
	To obtain the higher-order derivatives, we can note that the first two terms in the second block component equal
	\begin{equation*}
		\frac{\diff \bigl[\fun{\Lambda^{w^*\!,\mu,r}}(\rho)\bigr]_{\mathcal{I}} A_{\mathcal{I}} \fun{x^{w^*\!,\mu,r}}(\rho)}{\diff \rho},
	\end{equation*}
	to which the general Leibniz rule can be applied. Moving all but the first and last term of the resulting sum to the other side and using~\eqref{eqn:implicit-definition-derivative-c-lambda-rewrite} in the other direction, we obtain
	\begin{align*}
		\begin{split}
		&\mathopen{}\fun{J_{\fun{F}}}\bigl(\fun{w^{w^*\!,\mu,r}}(\rho)\bigr) \frac{\diff^{(q+1)} \fun{w^{w^*\!,\mu,r}}(\rho)}{\diff \rho^{(q+1)}} \\
		&\quad= - \sum_{i = 1}^q \binom{q + 1}{i} \begin{pmatrix} 0 \\ \frac{\diff^{(q+1-i)} \bigl[\fun{\Lambda^{w^*\!,\mu,r}}(\rho)\bigr]_{\mathcal{I}}}{\diff \rho^{(q+1-i)}} A_{\mathcal{I}} \frac{\diff^i \fun{x^{w^*\!,\mu,r}}(\rho)}{\diff \rho^i} \\ 0 \end{pmatrix}.
		\end{split}
	\end{align*}
	Setting $\mu = \mu_{k+1}$, $r = r_{k+1}$, $\rho = \lVert r_{k+1} \rVert$ and multiplying both sides with~$\lVert r_{k+1} \rVert^{q+1}$ and distributing this on the right-hand side according to the degree of differentiation, we obtain the desired relation.
\end{proof}

In \cite{Car05},~\cite{EV22} and~\cite{EV24} for linear, quadratic and general nonlinear programming problems respectively, when an extrapolation step is found not to be feasible to the implicit constraints, steps are defined that are equivalent to steps obtained by (partially) extrapolating to~$\rho = (1 - \theta) \lVert r_{k+1} \rVert$ instead of~$\rho = 1$ for $\theta \in [0, 1]$, where a (full) extrapolation step is obtained for $\theta = 1$. Considering the effect on the step size in the terms of the Taylor-series approximation, as
\begin{equation*}
	(1 - \theta) \lVert r_{k+1} \rVert - \lVert r_{k+1} \rVert = - \theta \lVert r_{k+1} \rVert
\end{equation*}
the point~$\fun{w_{k+1}^p}(\theta)$ resulting from taking a partial extrapolation step of order~$p$ can be obtained by scaling each~$\hat{w}_{k+1}^q$ with~$\theta^q$. Explicitly computing this step for $p = 2$ using the definition $\tilde{w}_{k+1}^q = \hat{w}_{k+1}^q / q!$, we get by \eqref{eqn:implicit-definition-derivative}~and~\eqref{eqn:w-hat-qp},
\begin{align*}
	\fun{J_{\fun{F}}}(w_{k+1}) \tilde{w}_{k+1}^1 &= 1/1 \cdot \fun{J_{\fun{F}}}(w_{k+1}) \hat{w}_{k+1}^1 = \operatorname{nml}(r_{k+1}) \cdot -\lVert r_{k+1} \rVert = -r_{k+1} \\
	&= -\fun{F^{\mu_{k+1}}}(w_{k+1}) \quad \text{and} \\
	\fun{J_{\fun{F}}}(w_{k+1}) \hat{w}_{k+1}^2 &= 1/2 \cdot - 2 \begin{pmatrix} 0 \\ \bigl[\hat{\Lambda}_{k+1}^1\bigr]_{\mathcal{I}} A_{\mathcal{I}} \hat{x}_{k+1}^1 \\ 0 \end{pmatrix} = - \begin{pmatrix} 0 \\ \bigl[\hat{\Lambda}_{k+1}^1\bigr]_{\mathcal{I}} A_{\mathcal{I}} \hat{x}_{k+1}^1 \\ 0 \end{pmatrix}
\end{align*}
and $\fun{w_{k+1}^2}(\theta) = w_{k+1} + \theta \tilde{w}_{k+1}^1 + \theta^2 \tilde{w}_{k+1}^2$. A variant can be obtained by scaling the extrapolation step with the same factor for all terms, to get in this setting $w_{k+1} + \theta \tilde{w}_{k+1}^1 + \theta \tilde{w}_{k+1}^2$ as next point as function of~$\theta$. An iterative algorithm taking at every iteration such a step while setting the barrier parameter to the mean complementarity has been shown in~\cite{Car09} not to be globally convergent for linear programming problems; the similarity with the Mehrotra predictor--corrector algorithm from~\cite{Meh91a} has been noted with the hope to gain understanding of the latter by studying the first. This resulted in the study in~\cite{CG08} of a variation on the Mehrotra predictor--corrector algorithm using multiple centrality correctors that uses different scalings for the different terms computed

\section{Local convergence of extrapolation step}
\label{ch:extrapolation-step-local-convergence}

With the extrapolation step stated, asymptotic properties of it derived and a general way of computing defined, in this section, local convergence of an algorithm taking extrapolation steps will be shown.

To analyze this, we will define the following algorithm in which an extrapolation step is always taken if such step is defined after a decrease of the barrier parameter and complemented if necessary by an inner minimization algorithm as Newton's method to find a point that fulfills the termination criteria.
\begin{myalgorithm}{Extrapolation primal--dual interior-point method}
	\label{alg:ipm}
	\begin{enumerate}[font=\upshape,leftmargin=*]
		\item
			Input: let $p \in \{1, \ldots, d - 2\}$, $\kappa \in (1, p + 1)$ and $\fun{\epsilon}$~and~$\fun{\varphi}$ be positive functions such that $\fun{\epsilon}(\mu_k) = \fun{\Theta}(\mu_k)$ and $\fun{\varphi}(\mu_k) = \fun{\Theta}(\mu_k^{\kappa})$. Choose $(x_0, \lambda_0) \in \mathbb{R}^{(n+m)}$ and $\mu_0 > 0$.

		\item
			Initialization: set the iteration index $k = 0$.

		\item
			Iteration: if $\fun{J_{\fun{F}}}(x_k, \lambda_k)$ is invertible, set $\bigl(\bar{x}_k, \bar{\lambda}_k\bigr) = \bigl(x^p_k, \lambda^p_k\bigr)$; otherwise, set $\bigl(\bar{x}_k, \bar{\lambda}_k\bigr) = (x_k, \lambda_k)$. Apply, if needed, an inner minimization method starting at~$\bigl(\bar{x}_k, \bar{\lambda}_k\bigr)$ for minimizing~\eqref{eqn:nlp} with complementarity perturbed by~$\mu_k$ until a point $(x_{k+1}, \lambda_{k+1})$ is found that fulfills~\eqref{eqn:primal-dual-termination}, i.e.,
			\begin{equation*}
				\lVert \fun{F^{\mu_k}}(w_{k+1}) \rVert \leq \fun{\epsilon}(\mu_k).
			\end{equation*}
			If a stopping criterion is not yet met, set $\mu_{k+1} = \fun{\varphi}(\mu_k)$, increment~$k$ with one and continue with a new iteration.

		\item
			Output: $(x_{k+1}, \lambda_{k+1})$ fulfilling a stopping criterion.
	\end{enumerate}
\end{myalgorithm}

The following theorem establishes convergence theory for this algorithm. It parallels Theorem~6.5 in~\cite{GOST01} for the case of $p = 1$ where the extrapolation step equals the Newton step and it shows a choice of parameters resulting in local convergence for the algorithm presented in~\cite{EV24} with convergence starting at a point close enough to the barrier trajectory for a barrier parameter that is sufficiently small.
\begin{theorem}
	\label{thm:extrapolation-convergence}
	Under the assumptions of Lemma~\ref{thm:existence-unique-barrier-trajectories-scalar}, including Assumption~\ref{asm:strict-complementarity}, let $\{w_k\}_{k \in \mathbb{N}}$ be a sequence of iterates generated by Algorithm~\ref{alg:ipm} without a stopping criterion such that there exists a subsequence indexed by~$\mathcal{K}$ for which $\{w_{k+1}\}_{k \in \mathcal{K}} \to w^*$. Then, the whole sequence of iterates~$\{w_k\}_{k \in \mathbb{N}}$ converges to~$w^*$ with ultimately no need for usage of the inner minimization method with componentwise R-convergence of order~$\kappa$ and componentwise Q-convergence of order~$\kappa$ for those components~$j$ for which $\bigl[\fun{\dot{w}^{w^*}}(0)\bigr]_j \neq 0$.
\end{theorem}
\begin{proof}
	By Lemma~\ref{thm:primal-dual-higher-order-termination-criteria-fulfilling}, for all $k \in \mathcal{K}$ sufficiently large, $w^p_{k+1} = w^{w^*\!, p}_{k+1}$ and by comparing~\eqref{eqn:primal-dual-termination} with~\eqref{eqn:primal-dual-extrapolation-termination}, we can see that $w^p_{k+1}$ will ultimately get accepted: $w_{k+2} = w^p_{k+1}$. By~\eqref{eqn:primal-dual-extrapolation-distance-optimality-upper-bound} and the convergence of~$\{w_{k+1}\}_{k \in \mathcal{K}}$, then also $\{w_{k+2}\}_{k \in \mathcal{K}}$ converges to~$w^*$. Inductively repeating this reasoning, it can be seen that the whole sequence of iterates~$\{w_k\}_{k \in \mathbb{N}}$ converges to~$w^*$ and that the extrapolation step is ultimately always accepted. Using~\eqref{eqn:primal-dual-extrapolation-distance-optimality-upper-bound}, it follows that
	\begin{equation*}
		[w_{k+2} - w^*]_j = \fun{O}(\mu_{k+1}) = \fun{O}(\mu_k^{\kappa}),
	\end{equation*}
	from which the R-convergence rate follows; more specifically, using~\eqref{eqn:primal-dual-extrapolation-distance-optimality} to argue about the rate of convergence for those components~$j$ such that $\bigl[\fun{\dot{w}^{w^*}}(0)\bigr]_j \neq 0$, we see that
	\begin{equation*}
		\frac{[w_{k+2} - w^*]_j}{\bigl([w_{k+1} - w^*]_j\bigr)^{\kappa}} = \fun{\Theta}\biggl(\frac{\mu_{k+1}}{\mu_k^{\kappa}}\biggr) = \fun{\Theta}\biggl(\frac{\mu_k^{\kappa}}{\mu_k^{\kappa}}\biggr) = \fun{\Theta}(1),
	\end{equation*}
	which finishes the proof.
\end{proof}

In the theorem above, the Q-convergence order is only established for those components of~$w$ for which the corresponding component in~$\fun{\dot{w}^{w^*}}(0)$ is nonzero, and it is a priori not clear that there always exist such components. Differentiating the equality
\begin{equation*}
	\fun{F^0}\bigl(\fun{x^{w^*}}(\mu), \fun{\lambda^{w^*}}(\mu)\bigr) = \begin{pmatrix} 0 & \mu e^T & 0 \end{pmatrix}^T
\end{equation*}
with respect to~$\mu$, we obtain among different equations
\begin{equation*}
	\begin{cases}
		0 = \fun{H}\bigl(\fun{x^{w^*}}(\mu), \fun{\lambda^{w^*}}(\mu)\bigr) \fun{\dot{x}^{w^*}}(\mu) - \fun{A}\bigl(\fun{x^{w^*}}(\mu)\bigr)^T \fun{\dot{\lambda}^{w^*}}(\mu); \\
		e = \bigl[\fun{\Lambda^{w^*}}(\mu)\bigr]_\mathcal{I} \fun{A_\mathcal{I}}\bigl(\fun{x^{w^*}}(\mu)\bigr) \fun{\dot{x}^{w^*}}(\mu) + \fun{C_\mathcal{I}}\bigl(\fun{x^{w^*}}(\mu)\bigr) \bigl[\fun{\dot{\lambda}^{w^*}}(\mu)\bigr]_\mathcal{I},
	\end{cases}
\end{equation*}
and only considering the components~$i \in I \cap \fun{\mathcal{A}}(x^*)$ of the bottom block that correspond to active inequality constraints evaluated for~$\mu = 0$,
\begin{equation*}
	\bigl[\fun{\lambda^{w^*}}(0)\bigr]_i \grad \fun{c_i}\bigl(\fun{x^{w^*}}(0)\bigr)^T \fun{\dot{x}^{w^*}}(0) = 1.
\end{equation*}
By strict complementarity, $\bigl[\fun{\lambda^{w^*}}(0)\bigr]_i \neq 0$, and we obtain
\begin{equation*}
	\grad \fun{c_i}\bigl(\fun{x^{w^*}}(0)\bigr)^T \fun{\dot{x}^{w^*}}(0) = 1 / \bigl[\fun{\lambda^{w^*}}(0)\bigr]_i,
\end{equation*}
from which we can conclude that $\fun{\dot{x}^{w^*}}(0) \neq 0$. Using the top block,
\begin{equation*}
	\fun{H}\bigl(\fun{x^{w^*}}(0), \fun{\lambda^{w^*}}(0)\bigr) \fun{\dot{x}^{w^*}}(0) = -\fun{A}\bigl(\fun{x^{w^*}}(0)\bigr)^T \fun{\dot{\lambda}^{w^*}}(0)
\end{equation*}
and since $\fun{H}\bigl(\fun{x^{w^*}}(0), \fun{\lambda^{w^*}}(0)\bigr)$ is nonsingular, also $\fun{\dot{\lambda}^{w^*}}(0) \neq 0$. Thus, as long as there is an active inequality constraint at a solution, there exists at least one component of the solution and a Lagrange multiplier vector for which Theorem~\ref{thm:extrapolation-convergence} establishes the Q-convergence order.

\section{Numerical experiments}
\label{ch:numerical-experiments}

Based on the acceleration framework outlined through Algorithm~\ref{alg:ipm}, results of numerical experiments on a proof-of-concept method to evaluate the performance will be discussed in this section. Covered by those tests are quadratic programming problems, as class of nonlinear problems for which the computations needed are of reduced complexity, as seen in Proposition~\ref{thm:w-hat-qp}.

Since Algorithm~\ref{alg:ipm} is an extrapolation framework in which the inner minimization is not specified, given that is asymptotically not needed by Theorem~\ref{thm:extrapolation-convergence}, the theoretical analysis applies to a wide range of practical algorithms. For the purpose of demonstrating the acceleration abilities, a practical variation on a baseline algorithm taking (partial) Newton steps in the inner minimization is studied. The algorithm is assumed to be given a starting point that is strictly feasible to the implicit constraints and uses outer and inner iterations. At each inner iteration, the $p$th-order extrapolation step is computed, as part of which the Newton step is obtained. To comply with the strict feasibility, both these steps are scaled-down if needed by the largest factor computed through a general formula such that the implicit constraints evaluate to at least the smallest strictly positive normal number in floating-point representation. After applying backtracking line search with the Armijo condition using the $2$-norm of the residual of perturbed optimality conditions as merit function on the possibly scaled Newton step, a comparison is made between the extrapolation step and the line-searched Newton step and the point at which the merit function evaluates to the smallest value gets chosen to start the next inner iteration.

Decreasing the barrier parameter at iteration~$k$ through $\mu_{k+1} = \min\{\mu_k^{\kappa},\mu_k/4\}$ and using $\lVert F^{\mu_k}(w_{k+1}) \rVert_{\infty} \leq \mu_k$ as inner termination criterion for a point~$w_{k+1}$, this algorithm can be seen to be practically compatible with Algorithm~\ref{alg:ipm} by setting~$\kappa$ to at most~$(p + 1)$. The algorithm has been implemented in the MATLAB platform for $p = 4$ and $\kappa = 4 + 1 = 5$, together with the unaccelerated baseline variant in which only Newton steps are taken for $\kappa = 1 + 1 = 2$ and the Mehrotra predictor--corrector algorithm. The Armijo line search is applied with parameter~$10^{-9}$. The stopping criteria used are those of the standard quadratic programming solver in MATLAB, which includes termination if no sufficient progress in the iterates is made, together with a timeout of 60 seconds.

Before passing problems to the algorithm, the problems are preprocessed. If lower and upper bounds are explicitly specified, these constraints are treated as general inequality constraints; if the lower bound equals the upper bound for a variable, the variable has a fixed value and the corresponding variable gets removed. To obtain a strictly feasible starting point, a linearly least squares solution to the equality constraints is first obtained using the normal equation; if the problem has no equality constraints, a primal solution with all components set to~$\epsilon = 0.4$ is used instead. For all inequality constraints that evaluate to a value strictly less than~$\epsilon$, shift variables are added to the formulation. The Lagrange multipliers to the equality constraints are set to~$1$ and the Lagrange multipliers to the inequality constraints are chosen such that the mean complementarity is~$5$.

The three algorithms have been applied to two sets of problems: the quadratic programming test set from~\cite{MM97} and randomly generated positivity-constrained problems. The first set consists of 138~problems of varying size, structure and density that have been collected from different sources. The randomly generated problems have positivity constraints on all variables and are generated with two parameters: the dimension~$n$ and conditioning~$t \geq 1$ of the problem. For a configuration with a given $n$~and~$t$, the objective function is set to~$x \mapsto \frac{1}{2} x^T H x + c^T x$ for~$H \in \mathbb{R}^{n \times n}$ a dense matrix with condition number~$t$ defined in terms of a random orthogonal matrix~$Q \in \mathbb{R}^{n \times n}$ generated through the procedure described in~\cite{Mez07}, a diagonal matrix~$T \in \mathbb{R}^{n \times n}$ with the diagonal components set to $\sqrt{t}$~and~$1/\sqrt{t}$ for the first and last and $\bigl(\sqrt{t}\bigr)^r$ otherwise for $r$ a realization of the uniform distribution on~$(-1, 1)$ through $H \triangleq Q T Q^T$ and a vector~$c \in \mathbb{R}^n$ whose components are realizations of the uniform distribution on~$(-1/2,1/2)$. The linear systems in the Jacobian of the residual of the perturbed optimality conditions at the current iteration point are solved using LU decomposition and, given the density of the problem descriptions, the coefficient matrix has been treated dense for the randomly generated problems while sparse for the other.

\begin{figure*}[t]
	\includegraphics[width=\columnwidth]{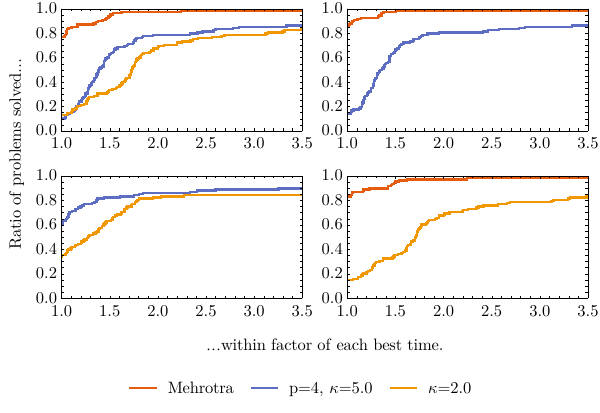}
	\caption{Performance profiles over 3~runs of solving the test set from~\cite{MM97} started at the solution output by the Mehrotra predictor--corrector algorithm upon the mean complementarity becoming smaller than~$1$, reporting only problems solved by at least one of the solvers.}
	\label{fig:qps-1-performance-profile}
\end{figure*}
In Figure~\ref{fig:qps-1-performance-profile}, a ranking between the different solvers is presented in the format of a performance profile as introduced in~\cite{DM02} based on the solution time for the diverse set of problems from~\cite{MM97}. To evaluate the performance of the extrapolation method as accelerator, the problems are initially solved by the Mehrotra predictor--corrector algorithm with the mean complementarity becoming smaller than~$1$ as termination criterion; the three solvers are then started at the output point and the recorded times concern this final phase. A timeout of 60~seconds is set for the initial solving and only problems for which a starting point could be obtained and that have been solved by at least one of the solvers are considered, which reduced the number of problems to~$108$. For the majority of the problems, the Mehrotra predictor--corrector solver continuing the initial phase outperforms the other two solvers. However, comparing the extrapolation solver to the baseline Newton solver, applying the extrapolation solver results on average in better solution times and the extrapolation step accelerates on average the baseline solver.

\begin{figure*}[t]
	\includegraphics[width=\columnwidth]{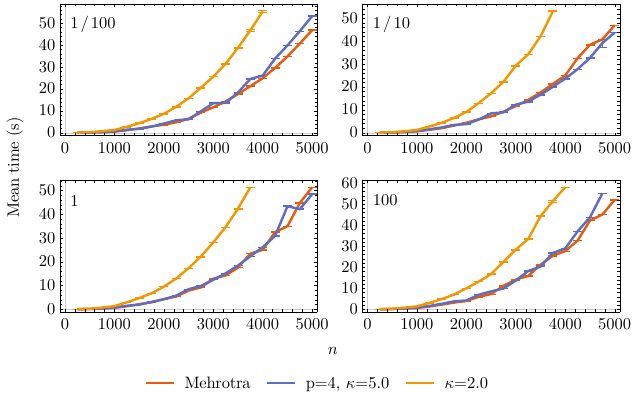}
	\caption{Mean times over 3~runs with error bars representing one standard deviation of solving randomly generated problems as described for $t$ equal to $n$ scaled with the number specified as plot label and $n$ of varying size.}
	\label{fig:random-0-condition-number-mean-time}
\end{figure*}
A comparison of solution time between the three solvers for the structured randomly generated positivity-constrained problems as global solver is presented in Figure~\ref{fig:random-0-condition-number-mean-time} for different problem sizes and conditionings that scale linearly with the problem size. For $t$ set to $n/100$,~$n/10$ or~$n$, the extrapolation solver seems to scale similar to the Mehrotra predictor--corrector solver, and is respectively slightly slower, slightly faster or comparable for larger problem sizes. Only for the relatively ill-conditioned problems with $t = 100 n$, the Mehrotra predictor--corrector solver seems to perform significantly better than the extrapolation solver. In all cases, the extrapolation solver outperforms the baseline solver. These observations suggests that for relatively well-conditioned problem, not only does the proof-of-concept solver accelerate the baseline solver, but it is also on a par with the Mehrotra predictor--corrector solver that performed well as global solver for the previous diverse test set.

\section{Discussion and future research}

We have shown how the concept of an extrapolation step in trajectory-following interior-point methods can be defined for a primal--dual method and how theoretically arbitrary fast convergence can be obtained by increasing the order of extrapolation. Of practical consideration, we note that the theoretical analysis assumes that the terms of the extrapolation step can be obtained with arbitrary precision: something that can not be satisfied for practical applications. As demonstrated for the case of quadratic programming, successive terms of the extrapolation step get computed by~\eqref{eqn:w-hat-qp} as solution of a linear system that depends on the previous terms; errors in the solution might therefore propagate to higher-order terms and the higher-order terms might be more sensitive to the quality of the solution of the linear systems. The quality of the extrapolation step might therefore deteriorate for problems with a higher condition number, as observed in the numerical experiments.

Theory for solving the linear systems arising in an interior-point method iteratively and inexactly has already been developed; see, e.g., \cite{Bel98} for the application on linear complementarity problems. In the light of the above, a study on the behavior for higher-order extrapolation methods could provide insight in a practically observable rate of convergence.

For up to second-order extrapolation, practical algorithms with complexity theory exists for extrapolation methods; see, e.g., \cite{ZZ95} for the application on linear complementarity problems. However, to the best of our knowledge, no such theory has been developed for the extrapolation of order higher than two, which could provide insight in the development of a practical global algorithm exploiting higher-order extrapolation for quadratic programming. For general nonlinear programming, initial findings on the performance have been reported in~\cite{EV24}, but no extensive study has been conducted.

\bibliography{ipm,references}

\end{document}